\documentclass[12 pt]{amsart}
\usepackage{amssymb}

\begin{document}
\title{Special tight closure}
\author{ Craig Huneke and Adela Vraciu}
\address {Department of Mathematics, University of Kansas, Lawrence,
KS 66045}
\email{huneke@math.ukans.edu; avraciu@math.ukans.edu}
\subjclass{13A35}
\date{February 14, 2001}
\thanks{The first author was partially supported by the National
Science Foundation.}
\begin{abstract} We study the notion of special tight closure of an
ideal and show that it can be used as a tool for tight closure computations.
\end{abstract}
\maketitle

\maketitle
\newcommand{\Hom}{\operatorname{Hom}}
\newcommand{\Ann}{\operatorname{Ann}}
\newcommand{\gr }{\operatorname{gr}}

\swapnumbers
\theoremstyle{plain}
\newtheorem{theorem}{Theorem}[section]

\newtheorem{prop}[theorem]{Proposition}
\newtheorem{lemma}[theorem]{Lemma}
\newtheorem{corollar}[theorem]{Corollary}
\newtheorem*{Corollary}{Corollary}

\theoremstyle{definition}
\newtheorem{obs}[theorem]{Observation}
\newtheorem{definition}[theorem]{Definition}
\newtheorem*{Definition}{Definition}
\newtheorem{example}[theorem]{Example}
\newtheorem*{notation}{Notation}

\newcommand{\li}{\tilde}
\newcommand{\aaa}{\mathfrak{a}}
\newcommand{\m}{\mathfrak{m}}
\newcommand{\param}{\underline{x}}
\newcommand{\tpar}{\underline{x}^{[t]}}
\newcommand{\tparq}{\underline{x}^{[tq]}}
\newcommand{\bs}{\boldsymbol}
\newcommand{\tx}{\noindent \textbf}
\newcommand{\ld}{\ldots}
\newcommand{\cd}{\cdots}
\newcommand{\q}{^{[q]}}

\newcommand{\spec}{I^{*sp}}
\newcommand{\f}{(f_1, \ld, \hat{f_i}, \ld, f_n)}
\newcommand{\eqq}{\Leftrightarrow}
\newcommand{\fs}{(f', f_1, \ld, \hat{f_i}, \ld , f_n)}
\newcommand{\cor}{^{<q>}}
\newcommand{\ins}{I_1\cap \ld \cap \hat{I_i} \cap \ld \cap I_n}
\newcommand{\xs}{x_1, \ld, x_d}
\newcommand{\xt}{(x_1^t, \ld, x_d^t)}
\newcommand{\xtq}{(x_1^{tq}, \ld, x_d^{tq})}
\newcommand{\product}{x_1\cdots x_d}

\newcommand{\ui}{\underline{i}}
\newcommand{\und}{\underline}
\newcommand{\arr}{\Rightarrow}
\newcommand{\lar}{\longrightarrow}

\section {Introduction}

Since the inception of tight closure theory, the problem of how to
compute and analyze the tight closure of an ideal has been of paramount
importance. The prevailing sense of how the tight closure of ideal $I$
compares to $I$ is that the extra elements needed to obtain the tight closure
of $I$ are much `deeper' in the ring than the generators of $I$. The first
result of this type is due to K.E. Smith \cite{Sm}: she proved that if $R$ is a normal
finitely generated algebra over a perfect field of positive characteristic,
$I$ is a homogeneous ideal generated by forms of degrees at least $\delta$
then every $x$ in the tight closure of $I$ but not in $I$ 
must have degree at least $\delta + 1$. The second author of this
paper was able to extend this result to non-homogeneous ideals (still in
a graded normal ring as above). This was done via a canonical decomposition
of the tight closure in terms of the ideal plus another piece called the
special tight closure.
 The purpose of this paper is to prove that the tight closure of an
 arbitrary ideal in a normal ring of positive characteristic with
 perfect residue field
 can be computed as the sum between the ideal and its special tight
 closure.
 Thus, the special tight closure can be envisioned as a technique for
 computing tight closure, or rather for imposing strong restrictions 
 on the set of elements that can be in the tight closure. As an
 application, we prove that if $(R, \m )$ is an excellent normal local ring with
perfect residue field and
$\gr_{\m} R$ is reduced, then for every ideal $I\subseteq \m^k$, 
$I^*\subseteq I + \m^{k+1}$, a direct generalization of Smith's result
to the local case. A consequence is that every ideal that lies
 between $\m^k$ and $\m^{k+1}$ is tightly closed.

One would like even more precise results. For example a special case of a  theorem
due independently to N. Hara \cite{Hara} and Mehta and Srinivas \cite{MS}
states that for large characteristic,
if $R$ is a Cohen-Macaulay graded ring with an isolated singularity and $x_1,...,x_n$ are
a homogeneous system of parameters of $R$ of degrees $d_1,...,d_n$, then
$$(x_1,...,x_n)^* = (x_1,...,x_n) + R_{\geq D}$$
where $D = d_1+...+d_n$ and where $R_{\geq D}$ is the ideal generated by
all forms of degree at least $D$. (See \cite{HS} for some background information
concerning this theorem.) 
What the best possible theorem might be remains a mystery.
We begin by introducing the relevant definitions and giving some background.

\begin{definition}\label{sp} Let $(R, \m )$ be a local Noetherian ring of
characteristic $p$, $p$ prime, 
 and let $I$ be an ideal.
 We say that an element $x \in R$ is in the {\it special tight
 closure} of $I$ if there exists $c \in R^0$ and a  fixed power of 
$p$, $q_0$, such that $cx^q \in \m ^{[q/q_0]}I\q$ for all $q \ge q_0$, or 
equivalently such that $x^{q_0}\in (\m I^{[q_0]})^*$.
\end{definition}

 In studying tight closure, it is natural to restrict our attention to
 a certain class of ideals, namely those ideals that are minimal among
 ideals having the same tight closure. 

\begin{definition} Let $(R, \m)$ be local characteristic $p$ ring.
An ideal $I$ of $R$ is $*$--independent if it can be generated by
elements $f_1, \dots, f_n$ (equivalently, for every minimal system of
generators $f_1, \dots, f_n$) such that for all $i=1, \dots, n$ we
have
$f_i \notin (f_1, \dots, f_{i-1}, f_{i+1}, \dots, f_n)^*$.
\end{definition}

 We note the following properties of special tight closure:
\begin{prop}\label{property} For a local ring
  $(R, \m)$ and an arbitrary ideal $I$,
the following hold:

\begin{enumerate} 
\item $\m I\subset I^{*sp}\subset I^*$.
\item If $I$ is $*$--independent, $I^{*sp}\cap I=\m I$.
\end{enumerate}
\end{prop}
\begin{proof}
 The first inclusion in 1. follows by choosing $q_0=1$ in Definition
 ~\ref{sp}, while the second inclusion follows from the definitions as
 well, since $cx^q \in m^{[q/q_0]}I\q  \Rightarrow cx^q \in
 I\q\Rightarrow x \in I^*$.
Part 2. is contained in Proposition 4.2 in \cite{V}.
\end{proof} 
 It is of considerable interest to identify situations in which $I^*$
 can be recovered from $I^{*sp}$. In Thm.4.4 in \cite{V}, it was proved that this can
 be accomplished if $R$ as the localization of a normal $\mathbb{N}$     -graded ring at
 the maximal homogeneous ideal.

 We now show that the assumption on the grading is unnecessary, thus
 the special tight closure technique for computing tight closure is
 available in a much larger class of rings.

The following result will be very useful. For a proof, see Proposition 2.4 in ~\cite{Ab}.

\begin{prop}\label{independent}
Let $(R, \m )$ be a excellent, analytically irreducible local ring of characteristic $p$, let $I$ be an ideal, and let $f\in R$. 
Assume that $f\notin I^*$; then there exists $q_0=p^{e_0}$ such that for all $q\ge q_0$ we have
$I\q :f^q \subset \m ^{[q/q_0]}$.
\end{prop}

\section{Main result}
\begin{theorem}
Let $(R, \m )$ be a characteristic $p$ local excellent normal ring, 
with perfect
residue field.
Then $I^*=I+I^{*sp}$ for every ideal $I$. In particular if $I$ is
$*$--independent, we have a direct sum decomposition
$$
\frac{I^*}{\m I}=\frac{I}{\m I}\oplus \frac{\spec }{\m I}.
$$
 \end{theorem}

\smallskip
Note that the second statement of the conclusion
 follows immediately from the first statement
 and the second part of ~\ref{property}.
Before beginning the proof we need several lemmas.

\begin{lemma}\label{suf}
Let $I=(f_1, \dots, f_n)$, and $f \in I^*$.
Assume that for all $i=1, \dots, n$ there exists an element $\alpha _i
\in R$ such that
$$
f_i \notin (f+\alpha _i f_i , f_1, \dots, f_{i-1}, f_{i+1}, \dots,
f_n)^*.
$$
Then $f\in I+I^{*sp}$.
\end{lemma}
\begin{proof}
Let $f'=f+\alpha _1 f_1 + \dots +\alpha _n f_n$; clearly $f' \in I^*$.
We claim that for all $i=1, \dots, n$, we have
$$f_i \notin (f', f_1, \dots, f_{i-1}, f_{i+1}, \dots, f_n)^*.$$
Since
$$(f', f_1, \dots, f_{i-1}, f_{i+1}, \dots, f_n)=(f+\alpha _i f_i,
f_1, \dots, f_{i-1}, f_{i+1}, \dots, f_n),
$$ this follows from the hypothesis.

Let $c\in R^0$ be such that $c(f')^q \in I\q$ for all $q$, and write
$$c(f')^q=a_1f_1^q +\dots +a_nf_n^q.$$
Then there exists fixed  $q_0=p^{e_0}$ such that for all $i$, we have
$$
a_i \in ((f')^q, f_1^q, \dots, f_{i-1}^q, f_{i+1}^q, \dots, f_n^q):f_i^q
\subset \m ^{[q/q_0]},$$
where the last containment follows 
 since we have shown that \newline
$f_i \notin (f', f_1, \dots, f_{i-1}, f_{i+1}, \dots, f_n)^*$
and can apply Prop.~\ref{independent}.
\end{proof}

The next lemma is a crucial step needed in the proof of our main
theorem.

\begin{lemma}\label{key}
Let $(R, \m)$ be excellent normal local ring of positive characteristic $p$,
and let $I=(f_1, \dots, f_n)$ be an arbitrary ideal.
If $f \in I^*$, there exists a test element $c$ and a power $q_0$ of
the characteristic such that
$$
cf^q \in cI\q +\m^{q/q_0}I\q 
$$
for all $q \ge q_0$.
\end{lemma}
\begin{proof}
There is no loss of generality in assuming that $f_1, \dots, f_n$ are
$*$--independent, because otherwise we can replace them by a 
$*$--independent subset, generating the same ideal up to tight
closure.

Since the ring is normal, the ideal defining its non-regular locus has
height at least two, and moreover one can choose two elements $c$, $d$
in this ideal, forming a regular sequence. Thm.6.2 in \cite{HH}
shows that we can replace $c$ and $d$ by some powers $c^n$, $d^m$ and
obtain a regular sequence consisting of test elements.

Write
$$cf^q=a_1f_1^q + \ld a_nf_n^q,\ \  \ \mathrm{and}\ \ \ \ df^q=b_1f_1^q+\ld +
b_nf_n^q,$$
with $a_1, \dots, a_n, b_1, \dots, b_n\in R$.
 Multiply the first equation by $d$ and the second one by $c$, then subtract the second equation from the first; we get:
$$
\sum_{i=1}^n (da_i-cb_i)f_i^q =0.
$$
Since $I$ is $*$--independent, this implies that there exists a $q_1$
(independent of $q$)
such that
$da_i -cb_i \in \m ^{q/q_1}\cap (c, d)$.
By the Artin-Rees lemma it follows that there exists a $q_2\ge q_1$, also
independent of $q$, such that 
$da_i -cb_i \in \m ^{q/q_2} (c, d)$, and therefore 
we can write
$da_i -cb_i=cu_i +dv_i$, with $u_i, v_i \in \m ^{q/q_2}$.
Because $c, d$ is a regular sequence on $R$, this implies that we can
write
$a_i=v_i+cg_i, \ \ b_i=-u_i +dg_i$. Recalling that $cf^q =\sum a_i f_i^q$, we
see that $c(f^q -\sum u_i f_i^q )\in \m ^{q/q_0}I\q$, which shows that
$cf^q \in cI\q +\m ^{q/q_0}I\q$.
\end{proof}

We now begin the proof of the theorem.
\begin{proof}
Let $I=(f_1, \dots, f_n)$ and $f\in I^*$. There is no loss of generality in assuming that $I$ is $*$--independent.
We show that one can find $\alpha _1, \dots, \alpha _n \in R$ such
that
the condition in Lemma ~\ref{suf} is satisfied. There is no loss of
generality in working with $i=1$. Assume by contradiction that for
every choice of $\alpha \in R$ we have $f_1\in(f+\alpha f_1, f_2,
\dots, f_n)^*$. Let $J_{\alpha }= (f+\alpha f_1, f_2,
\dots, f_n)$ and let $I_0=(f_2, \dots , f_n)$.

 Use lemma ~\ref{key} to write
\begin{equation}\label{eq1}
cf^q \equiv cu_qf_1^q \ \mathrm{mod} \ (\, I_0\q , \m ^{q/q_0}f_1^q\, )
\end{equation}
for some $u_q \in R$.

We claim that $u_{pq}\equiv (u_q)^p$ (mod $\m )$ for all $q\gg 0$.
To prove the claim, raise equation ~\ref{eq1} to the $p$th power to
obtain
\begin{equation}\label{eq2}
c^p f^{pq} \equiv c^p (u_q)^pf_1^{pq} \ \mathrm{mod} \ (\, I_0^{[pq]} , \m ^{pq/q_0}f_1^{pq}\, ).
\end{equation}
Using equation ~\ref{eq1} in which $q$ is replaced by $pq$, and then
multiplying by $c^{p-1}$, we get 
\begin{equation}\label{eq3}
c^pf^{pq}\equiv c^p u_{pq}f_1^{pq} \ \mathrm{mod} \ (\, I_0^{[pq]} ,
\m ^{pq/q_0}f_1^{pq}\, ).
\end{equation}
Comparing equations ~\ref{eq2} and ~\ref{eq3}, we get
$$
\left[ c^p(u_{pq}-(u_q)^p) -M \right] f_1^{pq} \in I_0^{pq},
$$
for all $q\gg 0$, where $M \in \m ^{pq/q_0}$.
We obtain that $$c^p(u_{pq}-(u_q)^p) -M\in I_0^{pq}:f_1^{pq}\subseteq \m^{pq/q_1}$$
for some fixed $q_1$ and all large $q$ since 
$f_1 \notin I_0^*$ (using  ~\ref{independent}).
Then there exists a fixed constant $k$ such that
 $$c^p(u_{pq}-(u_q)^p)\in \m^{pq/k}$$
 for all large $q$. 
Hence $u_{pq}-(u_q)^p\in \m^{pq/k}:c^p$ and the latter ideal is
proper for large enough $q$ since $c^p$ has finite order (it cannot be $0$).
Hence $u_{pq}-(u_q)^p \in \m $ for all $q \gg 0$.

Since the residue field $R/\m $ is perfect, we can choose $\alpha \in
 R$ such that $ \alpha ^q \equiv -u_q \ ( \mathrm{mod} \, \m)$ for all
 $q \gg 0$.

Using the assumption that $f_1 \in J_{\alpha }^*$ and Lemma
 ~\ref{key}, we get
$$
cf_1^q\in (\, I_0\q, c(f^q+\alpha ^q f_1^q), \m ^{q/q_0}f^q,
\m^{q/q_0}f_1^q\, ).
$$
Multiply by $c$ and use equation ~\ref{eq1} to obtain
$$
c^2f_1^q\in (\, I_0\q, c^2(u_q+\alpha ^q)f_1^q, \m ^{q/q_0}f_1^q\, ),
$$
and therefore we have
$$
f_1^q\left[ \, c^2\left(1-B(u_q+\alpha ^q)\right)-C\, \right] \in I_0\q,
$$
for some $B\in R$ and $C\in \m ^{q/q_0}$.

Note that the element
$c^2\left(1-B(u_q+\alpha ^q)\right)-C$ (which multiplies $f_1^q $ into
$I_0\q$) has bounded order (by the same reasoning as in the above paragraph)
 as $q\gg 0$ increases, because $u_q+\alpha
^q\in \m $ by the choice of $\alpha $. According to
contradicting the assumption that $I$ is $*$--independent. 
\end{proof}

\section{Applications}
As an application, we obtain a generalization of one of the main  results in
\cite{Sm}, showing that there is an explicit lower bound (depending on
the ideal) on the order of any
element in the tight closure of an ideal, provided that $(R, \m)$ is
complete normal and $\gr _{\m}R$ is reduced.

We wish to thank the referee for suggesting the present form of the results in 
this section.

\begin{theorem}\label{appl}
Let $(R, \m )$ be an excellent normal domain, with perfect residue
field. Assume that there exists a filtration $\mathcal{F}
= {F_k} $ consisting of $\m $-primary ideals, such that the graded ring $\gr \mathcal{ F} = \oplus F_k /F_{k+1}$
is reduced. If $I$ is any ideal such that $I\subset F_k$, then $I^*
\subset I+F_ {k+1}$.

In particular every ideal $I$ with the property that
$F_{k+1} \subset I \subset F_k$ for some integer value of $k$ is
tightly closed.

\end{theorem}
\begin{proof} 

Note that the assumption that $\gr \mathcal{F}$ is reduced implies that each of the ideals $F_k$ is integrally closed. If $u \in F_{k-1}\, \backslash \,  F_k$, it follows that $u^n \notin F_{n(k-1)+1}$ for all $n \ge 1$. This shows that 
an  equation of the form $u^n +a_1u^{n-1} +\cdots a_{n-1}u+a_n=0$, with $a_i \in F_k^{i}$, is impossible (because then $a_{n-i} u^i \in F_{(n-i)k}F_{i(k-1)}\subset F_{nk -i}\subset F_{n(k-1)+1}$ for all $i=0, \ldots, n-1$).

It is sufficient to show that $I^{*sp}\subset F_{k+1}$. Assume by
contradiction that there is a $u \in I^{*sp}$, $u \notin
F_{k+1}$. The assumption that $gr \mathcal{ F} $ is reduced implies that $u^{q_0} \notin F_{kq_0 +1}$ for every $q_0$.

Choose a fixed $c$ such that there exists a $q_0$ with
$cu^{q_0q} \in \m ^{q}I^{[q_0q]} $ for all $q \gg 0$. Since the ideals $F_1 \q $ are cofinal with $m ^q$, one can re-adjust $q_0$ so that, taking into account the fact that $I^{[q_0q]}\subset (F_{kq_0})^{[q]}$, we have 
\begin{equation}\label{eq4} cu^{q_0q} \in (F_{1+kq_0})\q. \end{equation}

Since $F_{kq_0+1}$ is integrally closed, we can choose $v$ a valuation such that $v(u^{q_0} ) < v(F_{kq_0}+1)$. Apply $v$ to
equation ~\ref{eq4}:
$$
v(c)+qv\left(u^{q_0}\right) \ge q v\left(F_{1+kq_0}\right).
$$
Dividing by $q$ and taking limits, we get
$v\left( u^{q_0} \right) \ge v\left(F_{1+kq_0}\right)$, which is a contradiction.

 \end{proof}
Note that Thm.2.2 in \cite{Sm} can be recovered as a particular case
of Proposition~\ref{appl}: Smith's result assumes that $R$ is a graded  normal
finitely generated ring over a perfect field, and concludes that 
every homogeneous element $x$ in the tight closure of a homogeneous ideal $I$ generated by
forms of degree at least $\delta$, but not in the ideal itself, must
have degree at least $\delta +1$.

\begin{corollar}
Let $(R, \m )$ be an excellent Cohen-Macaulay normal local domain with minimal multiplicity, i.e. $e(R)= \mathrm{edim}(R)-\mathrm{dim}(R)+1$, and with perfect infinite residue field. 
If ${\mathrm gr}_{\m}(R)$ is reduced, then $R$ is F-rational, i.e. all parameter ideals are tightly closed.
\end{corollar}
\begin{proof}
According to Thm. 4.2 (d) in \cite{HH}, it is enough to show that one parameter ideal is tightly closed. Let $\aaa $ be a parameter ideal which is a minimal reduction of $\m$. The minimal multiplicity assumption implies that $\m ^2 =\m \aaa$, since
$$
e(R)=l\left( \frac{R}{\aaa}\right) = l\left(\frac{R}{\m \aaa} \right) -l\left(\frac{\aaa}{\m \aaa }\right)
\ge 1+n -d,
$$ where $d=\mathrm{dim}(R), n=\mathrm{edim}(R)=\mu (\m)$, with equality if and only if $\m ^2 =\m \aaa $.
Thus we have $\m ^2 \subset \aaa \subset \m$; according to Thm. ~\ref{appl}, this implies that $\aaa $ is tightly closed. 
\end{proof}
Continuing along the same lines yields the following Proposition, which
may be useful for the study of so-called \it big ideals \rm (an ideal is
\it big \rm if every ideal containing it is tightly closed).

\begin{prop}\label{appl2} Let $R$, $\mathcal{F}$ be as in Thm. ~\ref{appl}; assume in addition that $gr \mathcal{F}$ is a domain. Let $I$ be an ideal of $R$
such that $F_{n+2}\subseteq I\subseteq F_n$ for some $n$.
Let $f_1,
\dots, f_t, l_1+g_1, \dots , l_s+g_s$ be a set of generators for $I$,
with $f_1, \dots, f_t, g_1, \dots, g_s \in F_{n+1}$, and $l_1, \dots,
l_s \in F_n \notin F_{n+1}$.
Let $I_0$ be the ideal of $gr {\mathcal F}$ generated by the classes of $l_0, \ld , l_s$ in $F_n /F_{n+1}$.

 If $I_0$ is tightly closed in $gr \mathcal{F}$, then
$I$ is also tightly closed.
\end{prop}

\begin{proof}  It is enough to show that $I^{*sp} \subset I$.  Let $u \in I^{*sp}$. By Thm. ~\ref{appl}, we may assume that $u \in F_{n+1} \notin F_{n+2}$, and 
 as in the proof of Thm. ~\ref{appl}, we can choose
           $c\in F_{\alpha } \notin F_{\alpha +1}$ and $q_0$
                                  such that
$cu^q \in F_{q/q_0}I\q$ for all $q\ge q_0$. 

In $gr {\mathcal F}$, this yields an equation of the form 
$$
\overline{cu^q } = \beta _1 \overline{l_1}^q +\cdots +\beta _s \overline{l_s}^q
$$
where the overline represents taking the class in $F_{\alpha +q(n+1)}/F_{\alpha +q(n+1)+1}$.

Since $I_0$ is tightly closed, this implies that $\overline{u}\in I_0$, where $\overline{u}$ represents the class of $u$ in $F_{n+1}/F_{n+2}$. Thus, 
$u \equiv  a_1l_0 + \cdots + a_s l_s \, \mathrm{mod}\  F_{n+2}$, where $a_1, \ldots, a_s \in F_1$. Since $a_i l_i \equiv a_i (l_i +g_i )\, \mathrm{mod}\  F_{n+2}$, it follows that $u \in I$.

\end{proof}

Note that this proposition allows us to reduce showing that the non-homogeneous ideal $I$ is tightly closed to showing that a homogeneous ideal $I_0$ is tightly closed.

Thm.. ~\ref{appl} and Prop. ~\ref{appl2} can be applied in  particular for the  case  $F_k =\m ^k$, with the corresponding assumptions on the graded ring $gr_\m R$.

If in Prop. ~\ref{appl2} one assumes instead that $R$ is graded normal domain, with $\m $ denoting the maximal homogeneous ideal, and the filtration is $F_k=\m^k$, the same proof shows  that $I^* \subset I+I_0^{*sp}$ (remove the assumption that $I_0$ is tightly closed).

 \bigskip

We wish to thank the referee for suggesting the present form of the results in Section 3. of the paper.

\end{document}